\documentclass[a4paper, 12pt]{article}
\usepackage{amsmath}%
\usepackage{amsthm}
\usepackage{amsfonts}%
\usepackage{amssymb}%
\usepackage[abbrev, nobysame]{amsrefs}
\usepackage{enumitem}
\usepackage{graphicx}
\usepackage{fullpage}
\usepackage{color}
\usepackage[utf8]{inputenc}
\usepackage[english]{babel}
\usepackage{soul}
\usepackage{cancel}

\newcommand{\eps}{\varepsilon}
\newcommand{\la}{\lambda}

\newcommand{\Rb}{\mathbb{R}}
\newcommand{\Vol}{\mathrm{Vol}}
\newcommand{\Div}{\mathrm{div}}
\newcommand{\inj}{\mathrm{inj}}
\newcommand{\out}{\mathrm{out}}
\newcommand{\gr}{\mathrm{gr}}
\newcommand{\inrad}{\mathrm{inrad}}
\newcommand{\CAP}{\mathrm{cap}}
\newcommand{\dif}{\mathrm{d}}
\newcommand{\me}{\mathrm{e}}

\newtheorem{theorem}{Theorem}[section]

\newtheorem{corollary}[theorem]{Corollary}

\newtheorem{definition}[theorem]{Definition}

\newtheorem{lemma}[theorem]{Lemma}
\newtheorem{notation}[theorem]{Notation}

\newtheorem{proposition}[theorem]{Proposition}
\newtheorem*{remark}{Remark}

\begin{document}
\title{The inner radius of nodal domains in high dimensions}
\author{Philippe Charron, Dan Mangoubi}
\date{\small{\sl\it In honor of Leonid Polterovich's 60th birthday}}
\maketitle
\begin{abstract}
    We prove that every nodal domain of an eigenfunction of the Laplacian of eigenvalue~$\lambda$ on a $d$-dimensional closed Riemannian manifold contains a ball of radius~$c\lambda^{-\frac{1}{2}}(\log\lambda)^{-\frac{d-2}{2}}$. This ball is  centered at a point at which the eigenfunction attains its maximum in absolute value within the nodal domain.
\end{abstract}
\section{Introduction and main result}
Let $(M,g)$ be a smooth closed Riemannian manifold of dimension $d$. 
Consider on~$M$ an eigenfunction~$u_\lambda$ of the positive Laplace-Beltrami operator $-\Delta_g$ corresponding to an eigenvalue~$\lambda$. A nodal domain~$\Omega_\lambda$ of~$u_\lambda$ is any connected component of the set $\{u_\lambda\neq 0\}$. It is well known (see~\cites{bruning78, berard-meyer}) that there exists a positive constant $c_{\mathrm{up}}=c_{\mathrm{up}}(M, g)$ independent of~$\lambda$ or $u_\lambda$ such that
every ball of radius bigger than $c_{\mathrm{up}}\lambda^{-1/2}$ contains a zero of~$u_{\lambda}$, i.e., the inner radius of~$\Omega_\lambda$ is bounded from above:
$$\inrad(\Omega_{\lambda})\leq c_{\mathrm{up}}\lambda^{-1/2}\ .$$
On the other hand, the Faber-Krahn inequality~\cites{faber, krahn} shows that the volume of every nodal domain~$\Omega_\lambda$ is bounded from below by $c_{FK}\lambda^{-d/2}$ for some positive constant $c_{FK}=c_{FK}(M)$.

One is naturally led to look for the largest positive number~$r$ such that every nodal domain contains a ball of radius~$r$. 
In two dimensions it is known that one can inscribe a ball of radius $c\lambda^{-1/2}$ in every nodal domain~\cite{mang-inrad}. This paper is concerned with lower bounds for the inner radius in higher dimensions. We give a lower bound which is optimal up to a logarithmic power factor at most.
Moreover, we show
\begin{theorem}\label{thm:innerradius}
Let $(M, g)$ be of dimension~$d$ at least three.
Let $x_{\max}\in\Omega_{\lambda}$ be a point where 
$|u_\lambda(x_{\max})|=\max_{\Omega_\lambda} |u_{\lambda}|$. Then
    $$B\left(x_{\max}, c_{\mathrm{lo}}\lambda^{-1/2}(\log\lambda)^{-(d-2)/2}\right)\subset\Omega_\lambda$$
    where $c_{\mathrm{lo}}=c_{\mathrm{lo}}(M, g)$ is a positive constant which depends only on~$(M, g)$.
\end{theorem}
Lower bounds of the form $\lambda^{-c(d)}$ on the inner radius  of nodal domains were obtained in~\cites{mang-inrad, mang-la} where $c(d)\to\infty$ as the dimension~$d$ increases. Under the assumption of a real analytic metric a lower bound with a constant power of the eigenvalue, namely~$c\lambda^{-1}$ was obtained by Georgiev in~\cite{geor}.

\paragraph{Idea of proof.} An important starting point is the existence of an almost inscribed ball~$B$ of radius~$r=\delta \lambda^{-1/2}$ for some small $\delta>0$, centered at~$x_{\max}$, guaranteed by   Georgiev and Mukherjee~\cite{geor-mukh} (cf.\ earlier works by Lieb and Maz'ya-Shubin~\cites{lieb, mazya-shubin} where existence of such a ball is proved without locating its center). This means that the relative volume of $\Omega_\lambda$ in~$B$  tends to~$1$ as~$\delta$ tends to~$0$. Our aim is to show that the complementary set $B\setminus\Omega_\lambda$ cannot approach the center too much, or that the gradient of~$u_\lambda$ is bounded, say in~$\frac{1}{4}B$. A standard elliptic gradient estimate shows that it would be enough to bound the values of~$|u_\lambda|$ in~$\frac{1}{2}B$.
The key argument shows that if in a ball~$b$ of radius $r/A$, where~$A$ is a suitable large number, the 
function $u_\lambda$ attains a large value, then its doubling index is also large in this ball. This lets us find  a finite sequence of~$A$ disjoint balls~$b_i$  where the corresponding doubling indices grow exponentially, to the extent that we get a contradiction with the Donnelly-Fefferman's growth bound. The argument is based on a recent Remez-type inequality for eigenfunctions due to Logunov and Malinnkova (see~\cites{logu-mali-icm, logu-mali-pcmi}). 

\paragraph{Comparison with two dimensions.}
Complex methods and a majorization principle 
comparing the harmonic measure of $B\setminus\Omega_\lambda$ in~$B$ with the eigenfunction allow to show in two dimensions that for a ball~$B=B(x_{\max}, \lambda^{-1/2})$ the complementary set $B\setminus\Omega$ cannot approach the center too much~\cite{mang-inrad}. However, the capacity estimates we apply in this paper, due to Lieb, Maz'ya-Shubin and Georgiev-Mukherjee~\cites{lieb, mazya-shubin, geor-mukh} are effective only in dimensions three and above. Interestingly,  Georgiev and Mukherjee~\cite{geor-mukh} prove this  capacity estimate by combining the above majorization principle with  heat flow methods.

\paragraph{Structure of the paper.}
The proof of the main Theorem~\ref{thm:innerradius} is given in \S\ref{sec:innerradius}.
Before that we present in \S\ref{sec:tools} the classical and more recent  tools needed for the proof.
In \S\ref{sec:geor-mukh} we reprove the existence of an almost inscribed ball with center located at a maximum point (Theorem~\ref{thm:capacity}) by methods which, in our view, are more oriented toward the PDE community than the Brownian motion tools used in~\cite{geor-mukh}.

\subsection{Acknowledgements}
We are very grateful to Stefano Decio for discussing aspects of this problem with us. We thank Yehuda Pinchover and Iosif Polterovich  for discussing with us some subtleties of the heat kernel estimates in \S\ref{sec:appendix} and for pointing our attention to Kannai's paper~\cite{kannai}.
We are very grateful to Sasha Logunov for pointing out the resemblance of our argument with Domar's proof of Levinson's $\log\log$ theorem (see~\citelist{\cite{koosis}*{Ch.\ VII, \S D.7}\cite{logu-papa}}). This significantly simplified the presentation of our proof.

The problem considered in this paper was originally suggested by Leonid Polterovich during DM's PhD thesis. DM wishes to express his deep gratitude to Leonid Polterovich for  his wise advice and for being a role model over the years.

PC was funded by the Zuckerman Postdoctoral Scholars Program at the Technion, and by the  NCCR SwissMAP, funded by the Swiss National Science Foundation (grant number 205607).

\section{Tool box: Growth of eigenfunctions, Remez-type inequality, almost inscribed balls, gradient estimate}\label{sec:tools}
\begin{notation}
Let $r_\inj$ denote the injectivity radius of $M$.
Let $B=B(x, r)$ be a geodesic ball in~$M$, where $2r<r_\inj$. We let $2B=B(x, 2r)$ be the concentric geodesic ball of twice the radius.
\end{notation}

\begin{definition}[The doubling index]
 Let $f$ be a function defined in a geodesic ball~$2B$. We let
$$N(f, B):=\log_2\frac{\sup_{2B} |f|}{\sup_{B} |f|}$$
\end{definition}

The following fundamental bound holds for eigenfunctions.
\begin{theorem}[Donnelly-Fefferman's growth bound \cite{don-fef88}]\label{thm:df-growth}
    Let~$(M, g)$ be a closed Riemannian manifold. Let $u_\lambda$ be a Laplace-Beltrami eigenfunction with corresponding eigenvalue~$\lambda$.
    Let $B=B(x, r)$ be a geodesic ball in~$M$ with $2r<\inj(M)$.
    Then 
    $$N(u_\lambda, B)\leq C_{DF}\sqrt{\lambda}\ ,$$
    for some $C_{DF}=C_{DF}(M, g)>0$.
\end{theorem}

Another tool we apply is the following Remez-type inequality. It is a version of the three balls inequality for wild sets. This is a recent result, due to Logunov and Malinnikova,  which was conjectured by Landis and Donnelly and Fefferman, who gave an idea on how such growth bounds can be used for the study of nodal sets~\cite{don-fef90}.
\begin{theorem}[Remez-type inequality for solutions of elliptic equations~\cite{logu-mali-icm}]\label{thm:remez-elliptic}
    Let $h$ be a solution to a linear elliptic equation in divergence form 
    $\Div (A\nabla h) = 0$ in a ball~$2B\subseteq\Rb^d$. 
    For any measurable subset $E\subset B$
    $$\sup_B |h|\leq C_1\sup_E |h| \left(C_1\frac{\Vol(B)}{\Vol(E)}\right)^{C_1 N(h, B)}\ ,$$
    for some $C_1=C_1(A)>0$.
\end{theorem}

We will apply the immediate corollary below, adapted for eigenfunctions.
\begin{corollary}[Remez-type inequality for eigenfunctions in small scales]\label{cor:remez-eigen}
Let $(M, g)$ be a closed Riemannian manifold. Let $u_\lambda$ satisfy the eigenvalue equation $-\Delta_g u_\lambda = \lambda u_\lambda$ in a geodesic ball $2B=B(x, 2r)$ with $r<\lambda^{-1/2}$. Then,
 $$\sup_B |u_\lambda|\leq C_R\sup_E |u_\lambda| \left(C_R\frac{\Vol(B)}{\Vol(E)}\right)^{C_R N(u_\lambda, B)+C_R}\ ,$$
 for some $C_R=C_R(M, g)>0$.
\end{corollary}
\begin{proof}
    Consider the harmonic function $h(t, x):=e^{t\sqrt{\lambda}} u_\lambda(x)$ defined on the~$d+1$ dimensional ball
    $B^{d+1}((0, x), 2r)$ in the product Riemannian manifold $\Rb\times M$.
    Then, $$N(u_\lambda, B(x, r))-C\leq N(h, B^{d+1}((0, x), r))\leq N(u_\lambda, B(x, r))+C$$ 
    Set $E^{d+1}=\{(t, y)\in B^{d+1}((0, x), r) |\ y\in E\}$.
    Apply Theorem~\ref{thm:remez-elliptic} to $h$ with~$E$ replaced by $E^{d+1}$ and $B$ replaced by $B^{d+1}$.
\end{proof}

The following ingredient in our proof, which is due to Georgiev and Mukherjee, guarantees a large almost inscribed ball in a nodal domain, centered at a maximum point of the eigenfunction in the nodal domain.
\begin{theorem}[An almost inscribed ball \cite{geor-mukh}]\label{thm:capacity}
Let $(M, g)$ be a closed Riemannian manifold of dimension~$d$ at least three. Let $u_\lambda$ be an eigenfunction and  $\Omega$ a corresponding nodal domain.  Let~$x_{\max}\in\Omega$ be a point where
$|u_\lambda(x_{\max})|=\max_\Omega |u_\lambda|$. Then, for  any geodesic ball $B_\delta=B(x_{\max}, \delta\lambda^{-1/2})$
$$\frac{\Vol(B_\delta\setminus\Omega)}{\Vol(B_\delta)}\leq C_\out\delta^{2d/(d-2)}$$
for some $C_\out=C_\out(M, g)>0$.
\end{theorem}
\begin{remark} Lieb, Maz'ya and Shubin~\cites{lieb, mazya-shubin} proved the existence of an almost inscribed ball of radius~$\sim\lambda^{-1/2}$ in a bounded domain in the Euclidean domain with first Dirichlet eigenvalue~$\lambda$. Their arguments do not give information about its center.
\end{remark}

Finally, we recall the following classical gradient estimate for solutions of linear elliptic equations.
Let 
$$L = a^{ij}\partial_i\partial_j+b^i\partial_i+c $$
be a uniformly elliptic operator of second order in the  unit ball.
Assume that 
$$\|a^{ij}\|_{C^2(B_1)}+\|b^i\|_{C^1(B_1)}+\|c\|_{C(B_1)}\leq K$$
and that 
$$\forall \xi\in\Rb^d\quad a^{ij}\xi_i\xi_j \geq \Lambda |\xi|^2$$
for some $\Lambda>0$. We have
\begin{theorem}[A gradient estimate {\cite{gil-tru}*{Theorem 6.2}}]
\label{thm:gradient}
Let $h$ satisfy $Lh=0$ in the ball~$B_r=B(0, r)$ where $0<r<1$.
Then,
$$|\nabla h(0)|\leq \frac{C_2}{r}\sup_{B_r} |h|\ ,$$
for some $C_2=C_2(K, \Lambda, d)>0$.
\end{theorem}

Adjusted for eigenfunctions using the same lifting trick as in the proof of Corollary~\ref{cor:remez-eigen} we get
\begin{corollary}[A gradient estimates for eigenfunctions]\label{cor:gradient-eigen}
Let $(M, g)$ be a closed Riemannian manifold and let $u_\lambda$ satisfy $-\Delta_g u_\lambda = \lambda u_\lambda$ in a geodesic ball $B(x, r)$,
where $r<\lambda^{-1/2}$. Then,
$$ \sup_{B(x, r/2)}|\nabla u_\lambda|\leq \frac{C_\gr}{r}\sup_{B(x, r)} |u_\lambda|$$
for some $C_\gr=C_\gr(M, g)>0$.
\end{corollary}

\section{Proof of the main Theorem~\ref{thm:innerradius}}\label{sec:innerradius}
Let $\Omega_\lambda$ be a nodal domain of $u_\lambda$ and $x_{\max}$ be a point where $|u_\lambda|$ attains its maximum in~$\Omega_\lambda$. We  re-normalize $u_\lambda$ so that $u_\lambda(x_{\max}) =1$.

Identify the geodesic ball $B(x_{\max}, r_\inj)$ by means of the exponential map at~$x_{\max}$ with an open ball of radius~$r_\inj$ in Euclidean space, centered at~$0$. Assume that~$\lambda>1$  and let~$B_\delta$ be a ball of radius $\delta\lambda^{-1/2}$ centered at $0$.   
\begin{lemma}
\label{lem:partition-not-too-fine2}
There exists $c_1=c_1(M, g)>0$ such that
if $A\leq c_1 \delta^{-2/(d-2)}$ then for any ball $b \subset B_\delta$ of radius $\delta\lambda^{-1/2}/A $, $$ \frac{\Vol(b \cap \Omega_\lambda)}{\Vol(b)}\geq  \frac{1}{2}\ . $$
\end{lemma}
\begin{proof}
Otherwise, for any $c_1>0$, some $A$ as above and some ball~$b\subset B_\delta$ of radius $\delta\lambda^{-1/2}/A$ we get the following inequalities.
$$
\frac{\Vol(B_\delta\setminus
\Omega_\lambda)}{\Vol(B_\delta)}\geq\frac{\Vol(b\setminus\Omega_\lambda)}{\Vol(B_\delta)}>
\frac{1}{2}\frac{\Vol(b)}{\Vol(B_\delta)}\geq \frac{1}{2}c A^{-d} \geq \frac{1}{2}c c_1^{-d}\delta^{2d/(d-2)}\ .$$
The preceding inequality  contradicts Theorem~\ref{thm:capacity} for $c_1>0$ small enough depending on~$C_\out, M, g$ only.
\end{proof}

\begin{lemma}\label{lem:doublingremez}
 There exist $T>1, a>0$ which only depend on $(M,g)$, such that for any ball~$b \subset B_{r_\inj/2}$ which satisfies $\Vol(b \cap \Omega_\lambda)/\Vol(b) \geq 1/2$  and $\sup_b |u_\lambda| \geq T$ , it holds that 
    \begin{equation}
        \sup_{2b} |u_\lambda| \geq (\sup_{b} |u_\lambda|)^{1+a}\ .
    \end{equation}
\end{lemma}

\begin{proof}

    By Corollary \ref{cor:remez-eigen},    
    $$\sup_b |u_\lambda|\leq C_R\sup_{b \cap \Omega_\lambda} |u_\lambda| \left(C_R\frac{\Vol(b)}{\Vol( b \cap \Omega_\lambda)}\right)^{C_R N(u_\lambda, b)+C_R} \leq C_R (2C_R)^{C_R N(u_\lambda, b)+C_R}$$ 
    
    This implies that  $N(u_\lambda, b) \geq c'\log_2 \sup_b |u_\lambda| - c''\geq (c'\log_2 \sup_b |u_\lambda|)/2$, where
    the last inequality is true if $\log_2 \sup_b|u_\lambda| \geq T$, with~$T$ large enough.
   The preceding inequality amounts to
   $$\log_2{\frac{\sup_{2b} |u_\lambda|}{\sup_{b} |u_\lambda|}} \geq\frac{c'}{2}\log_2(\sup_b |u_\lambda|)\, $$ 
   or, $$\sup_{2b} |u_\lambda| \geq \sup_{b} |u_\lambda|^{1+c'/2}\ .$$
\end{proof}

We now apply Lemma~\ref{lem:doublingremez} iteratively on small balls in $B_\delta$ to find a ball with a large doubling index.
\begin{lemma}\label{lem:doublinginduction}
    Suppose $\sup_{B_{\delta/2}} |u_\lambda| > T$, where $T$ is as in Lemma~\ref{lem:doublingremez}. Set $A=c_1\delta^{-2/(d-2)}$. Then 
    $$N(u_\lambda, B_{\delta/2}) \geq e^{c_2A}\ .$$
\end{lemma}
\begin{proof}

Consider a longest sequence of balls $b_k \subset B_\delta$ with the following properties
\begin{enumerate}[label=(\roman*), align=left]
    \item Each ball $b_k$ is of radius $\delta \lambda^{-1/2} / A$.
\item The ball $b_1$ is centered at a point where $|u_\lambda|$ attains its maximum in the closure of $B_{\delta/2}$.
 \item The center of $b_{k+1}$ is located at a point where $|u_\lambda|$ attains its maximum in the closure of $2b_k$.
\end{enumerate}
From Lemma~\ref{lem:partition-not-too-fine2} we know that for all~$k$
$$ \frac{\Vol(b_k \cap \Omega_\lambda)}{\Vol(b_k)} \geq  \frac{1}{2}\ ,$$
 and so, we may apply Lemma \ref{lem:doublingremez} to each ball $b_k$ to obtain that 
$$\sup_{b_{k+1}} |u_\lambda|\geq \sup_{2b_k} |u_\lambda|\geq \sup_{b_k} |u_\lambda|^{1+a}  \, .$$ 
In particular, we can find  at least $\lfloor(\delta\lambda^{-1/2}/2)/(2\delta\lambda^{-1/2}/A)\rfloor=\lfloor A/4 \rfloor$ distinct balls in the sequence. It follows
\begin{equation*}
    \sup_{B_\delta} |u_\lambda|\geq \sup_{b_{\lfloor A/4\rfloor}} |u_\lambda|\geq \sup_{b_1} |u_\lambda|^{(1+a)^{\lfloor A/4 \rfloor}}
     \geq  \sup_{B_{\delta/2}}|u_\lambda|^{{(1+a)^{\lfloor A/4 \rfloor}}} \ .
\end{equation*}
Finally,
$$N(u_\lambda, B_{\delta/2})\geq {(1+a)^{\lfloor A/4 \rfloor}-1}\geq \me^{c A}\ .$$
\end{proof}

The preceding lemma shows fast growth in the ball $B_\delta$. When confronted with the Donnelly-Fefferman growth bound in Theorem~\ref{thm:df-growth} we obtain a bound on the size of $u_\lambda$.
\begin{proposition}\label{prop:maximaldoublingindex}
There exists $c_3=c_3(M, g)>0$ such that if  
$\delta=c_3(\log\lambda)^{-\frac{d-2}{2}}$, 
then $$\sup_{B_{\delta/2}} |u_\lambda|<T\ .$$
\end{proposition}
\begin{proof}
Fix $c_3>0$ (and~$\delta$) so that for~$A$ as in Lemma~\ref{lem:doublinginduction} we have
$c_2 A> \log (C_{DF}\sqrt{\lambda})$. If $\sup_{B_{\delta/2}}|u_\lambda|\geq T$
we obtain from Lemma~\ref{lem:doublinginduction} that 
  $$N(u_\lambda, B_{\delta/2})> C_{DF}\sqrt{\lambda}\ ,$$
  contradicting Theorem~\ref{thm:df-growth}.
\end{proof}

With the preceding proposition we combine the gradient estimate to finish the proof of Theorem~\ref{thm:innerradius}.
\begin{proof}[Proof of Theorem~\ref{thm:innerradius}]
  Set $\delta$ as in Proposition~\ref{prop:maximaldoublingindex}. Then, $\sup_{B_{\delta/2}}|u_\lambda| \leq T$. By Corollary~\ref{cor:gradient-eigen}, 
  $\sup_{B_{\delta/4}} |\nabla u_\lambda| \leq \frac{2C_\gr T}{\delta}$. Since $u_\lambda(x_{\max}) = 1$, it follows that we can find  a  ball of radius~$\frac{\delta\lambda^{-1/2}}{c}$ where $u_\lambda$ is positive. 
\end{proof}

\section{Proof of Theorem~\ref{thm:capacity}: A capacity estimate and an almost inscribed ball}\label{sec:geor-mukh}
We  give a proof of Theorem \ref{thm:capacity}, due to Georgiev-Mukherjee~\cite{geor-mukh}. 
One can estimate the capacity of  a condenser $(B\setminus\Omega, B)$ by considering the heat flow in  
 $B\cap\Omega$ with boundary conditions set to one on $B\cap\partial\Omega$ and zero elsewhere,  while setting the initial condition to zero (see~\cite{grig-saloff}). The main idea in~\cite{geor-mukh} (cf. \cite{mang-inrad}*{\S3}) is then to compare the heat flow in~$B\cap\Omega$ with the heat flow in~$\Omega$ starting  from an eigenfunction, solved explicitly. One gets an explicit estimate of the heat at a maximum point of the eigenfunction in~$\Omega$.
 
 The actual implementation of this elegant idea in~\cite{geor-mukh} is done through Brownian motion techniques, and some points are only sketched. We  give here an implementation of this idea which is more oriented toward the PDE community. The tools used below  replacing the Brownian motion arguments are Li-Yau Gaussian upper bounds,
 Gaussian lower bounds and the principle of not feeling the boundary.

\subsection{Estimating the capacity via the heat flow}
In this section we recall Theorem~\ref{thm:capacity-upper-bound} due to Grigor'yan-Saloff-Coste~\cite{grig-saloff}*{Theorem 3.7},
which gives a way to estimate the capacity of a condenser by considering the heat flow in it and taking one measurement of temperature at a point. We only  treat the case of a smooth condenser, as the more general case, treated using the tools of Potential Theory (see~\cite{grig-saloff}), is not needed for our purpose.

\subsubsection{Recalling the heat equation and the notion of capacity}
Let $U\subset M$ be a chart with smooth boundary, and let $K\subset U$ be a compact set with smooth boundary.
The heating capacity $\CAP(K, U)$ of the condenser $(K, U)$  is the heat rejected per  unit time by the condenser  when the temperature  drops by one unit. It is defined 
(see~\cite{mazya}*{\S2.2.1} or~\cite{hein-kilp-mart}*{Ch. 2}) by
\begin{definition}[Capacity]
We set 
$$\CAP(K, U)=\inf\left\{\int_{U} |\nabla v|^2 \, dx\ :\, v\in C_c^1(U) \text{ and } v\geq 1 \text{ on } K\right\}\ .$$
\end{definition}
Let $\psi_{K, U}^{eq}$ be the solution to the Dirichlet problem
$$\left\{\begin{array}{ll}
  \Delta u=0, & \text{ in } U\setminus K,\\
  u=1,  & \text{ on } \partial K,\\ 
  u=0, &  \text{ on } \partial U.
\end{array}\right.$$
A classical variational argument and Green's identity show that
\begin{equation}\label{eq:capacity-flux}
    \CAP(K, U)=\int_{U\setminus K} |\nabla \psi_{K, U}^{eq}|^2\, \dif x =
-\int_{\partial K}  \partial_\nu \psi_{K, U}^{eq}\,\dif\sigma
\end{equation}
where $\nu$ is the outward unit normal with respect to $K$.

We let $\psi_{K, U}(t, x)$ denote the heat flow in the condenser $(K, U)$ with initial conditions set to zero, i.e., it solves 
$$\begin{cases}
     \partial_t u=\Delta_x u &\text{ in } \Rb_+\times (U\setminus K), \\
     u(0, x)=0 & \text{ for all } x\in U\setminus K, \\
     u(t, y)=0 & \text{ for all } y\in \partial U \text{ and } t>0,\\
     u(t, y)=1 & \text{ for all } y\in \partial K \text{ and } t>0.
\end{cases}
$$

We recall that for a smooth open subset $V\subset M$ a solution $u(t, x)$ to the heat equation with zero boundary condition and initial condition $u_0(x)$ is given by a smooth heat kernel, denoted by $p_V(t, x, y)$:
$$u(t, x)=\int_V p_V(t, x, y)u_0(y)\, dy\ .$$

\subsubsection{An upper bound on capacity via the heat flow}
The following theorem gives an upper bound on the capacity in terms of $\psi_{K, U}(t, x)$ and a heat kernel lower bound:
\begin{theorem}[{\cite{grig-saloff}*{Theorem 3.7}}]\label{thm:capacity-upper-bound}
$$\CAP(K, U) \int_0^t \inf_{y\in\partial K} p_{U}(s, x, y ) \,ds \leq \psi_{K, U}(t, x)$$
for all $x\in U\setminus K$ and $t>0$.
\end{theorem}

For the proof we start with
\begin{lemma}
	For all $t > 0$ and $x \in U\setminus K$, 
	\begin{equation} \label{diffheat}
		\psi_{K, U}^{eq}(x) - \psi_{K, U}(t,x) = \int_{U\setminus K}  p_{U\setminus K}(t, x,y)\psi_{K, U}^{eq}(y) dy\ .
	\end{equation}
\end{lemma}
\begin{proof}
Indeed, both sides of the equation satisfy the heat equation in~$U\setminus K$
with  the same initial condition and with  boundary condition set to zero.
\end{proof}

On the other hand the equilibrium temperature $\psi_{K, U}^{eq}$
is expressed via the Green kernel as:
\begin{lemma} For all $x\in U\setminus K$
	\begin{equation}\label{eq:Green1}
			\psi_{K, U}^{eq}(x) = -\int_{\partial K} G_U(x,y) \partial_\nu \psi_{K, U}^{eq}(y)\, d\sigma(y)\ .
	\end{equation}
	where $\nu$ is the outer unit normal with respect to~$K$.
\end{lemma}
\begin{proof}
Green's identity in $U\setminus K$ shows,
as $\psi_{K, U}^{eq}$ is harmonic, that 
$$
		\psi_{K, U}^{eq}(x) = 
		 -\int_{\partial K}  G_U(x,y) \partial_\nu \psi_{K, U}^{eq}(y)\, d\sigma(y) 
   +\int_{\partial K} \partial_\nu G_U(x,\cdot)(y) \psi_{K, U}^{eq}(y)\,  d\sigma(y)\ .
$$

Observe now that the second integral term vanishes for $x\in U\setminus K$:
$$
	\int_{\partial K} \partial_{\nu} G_U(x,\cdot)(y)\, d\sigma(y)= \int_{K} \Delta_y G_U(x,y)\, dy
	= 0\ .
$$
\end{proof}
\begin{remark}
One already sees how capacity might arise by
comparing the right hand side of~\eqref{eq:Green1} with the expression in~\eqref{eq:capacity-flux}.
\end{remark}

\begin{proof}[Proof of Theorem~\ref{thm:capacity-upper-bound}]
We combine the preceding two lemmas, the connection between the Green kernel and the heat kernel (see~\cite{chavel}*{Ch.\ VII, p. 177}) and the maximum principle
 in order to estimate $\psi_{K, U}^{eq}(x)-\psi_{K, U}(t,x)$.

\begin{equation}\label{ineq:distance-from-equilibrium}
\begin{split}
	\psi_{K, U}^{eq}(x) - \psi_{K, U}(t,x) &= \int_{U\setminus K} p_{U\setminus K}(t,x,y) \psi_{K, U}^{eq}(y)\, dy\\
 &= -\int_{U\setminus K} \int_{\partial K} p_{U\setminus K}(t,x,y) G_U(y,z) \partial_{\nu} \psi_{K, U}^{eq}(z)\, d\sigma(z)\, dy\\
	&= -\int_{U\setminus K} \int_{\partial K} \int_0^\infty p_{U\setminus K}(t,x,y) p_U(s,y, z) \partial_\nu \psi_{K, U}^{eq}(z)\, ds\, d\sigma(z)\, dy \\
	&\leq  -\int_{\partial K} \int_0^\infty \left(\int_{U} p_U(t,x,y) p_U(s,y,z)\, dy \right) \partial_\nu \psi_{K, U}^{eq}(z)\,  ds\,  d\sigma(z)\\
	&=   - \int_{\partial K} \int_0^\infty  p_U(t+s,x,z)  \partial_\nu \psi_{K, U}^{eq}(z)\, ds\, d\sigma(z)\\
 &=- \int_{\partial K} \left(\int_t^\infty  p_U(s,x,y) \, ds\right) \partial_\nu \psi_{K, U}^{eq}(y)\, d\sigma(y)\ .
 \end{split}
\end{equation}

On the other hand, by \eqref{eq:Green1} we can write:
\begin{multline}\label{eq:Green2}
	\psi_{K, U}^{eq}(x)= -\int_{\partial K} G_U(x,y)\partial_{\nu} \psi_{K, U}^{eq}(y)\, d\sigma(y)\\
	= -\int_{\partial K} \left(\int_0^\infty p_U(s,x,y) \, ds\right)\partial_\nu \psi_{K, U}^{eq}(y)\, d\sigma(y)\ .
\end{multline}

Therefore, combining \eqref{ineq:distance-from-equilibrium} and \eqref{eq:Green2}, we have the following lower bound for $\psi_{K, U}(t,x)$:

\begin{multline*}\label{eq41?}
	\psi_{K, U}(t,x) 
	 \geq -\int_{\partial K} \left(\int_0^t p_U(s,x,y)\, ds\right)\partial_\nu \psi_{K, U}^{eq}(y)\, d\sigma(y)\\
	 \geq -\left(\int_0^t \inf_{y \in\partial K} p_U(s,x,y)\, ds\right) \int_{\partial K}\partial_\nu \psi_{K, U}^{eq}(y)\, d\sigma(y)\\
	 = \left(\int_0^t \inf_{y\in \partial K} p_U(s,x,y)\, ds\right) \CAP(K, U)\ .
\end{multline*}

\end{proof}
\subsection{Estimating the heat kernel from below}\label{sec:heat-kernel-below}
In order to apply the general capacity upper bound via the heat flow (Theorem~\ref{thm:capacity-upper-bound}) we need to know lower bounds on the heat kernel. We cover~$M$ by a finite number of geodesic balls~$\{B(x_k, r_0/4)\}_k$, such that $\{B(x_k, r_0)\}_k$ are strongly convex geodesic balls.
We set $U=B(x_j, r_0)$, and~$K\Subset U$, where the point~$x_j$ and the subset~$K$ will be specified in~\S\ref{sec:majorization}.
\begin{proposition}\label{prop:capacity-temperature}
	There exists a positive constant $C$ such that 
 for all \mbox{$0<r<r_0$} and $x\in U\setminus K$ such that $K\subset B(x, r)$
	\begin{equation*}
 		\CAP(K, U)\leq C\psi_{K, U}(r^2, x)  r^{d-2}\ .
	\end{equation*}
\end{proposition}

\begin{proof}
Recall the Gaussian lower bound for the Dirichlet heat kernel in~$U$,
which immediately follows from Theorems~\ref{thm:heat-kernel-lower-bd} and~\ref{thm:not-feeling-bdry}: 
 There exists $t_0>0$ such that for all $0<t<t_0$, and all $r>0$ small enough 

$$\int_0^t \inf_{y\in \partial K} p_U(s, x, y)\, ds \geq \int_0^t Cs^{-d/2}e^{-C r^2/s}\, ds\ .$$
For $t=r^2$ we conclude

\begin{equation}\label{ineq:lower-bound-truncated-green}
\int_{0}^{r^2} \inf_{y\in\partial K} p_U(s, x, y)\, ds\geq C r^{2-d}\ .
\end{equation}

Combining Theorem~\ref{thm:capacity-upper-bound} with inequality~\eqref{ineq:lower-bound-truncated-green}, we conclude the desired estimate. 
\end{proof}

\subsection{Comparing $\psi_{K, U}$  to a flow of an eigenfunction}\label{sec:majorization}
In this section $\Omega$ is a nodal domain of an eigenfunction~$u_\lambda$, where $u_\lambda$
is positive in~$\Omega$. The point~$x_{\max}\in\Omega$ is a maximal point of $u_\lambda$ in $\Omega$. 
For $0<\delta<1$ and $\lambda>\lambda_0$ large enough we let $K\subset B(x_{\max}, \delta\lambda^{-1/2})\setminus\Omega$ be any compact smooth subset.
We fix $U=B(x_j, r_0)$ to be one of the strongly convex geodesic balls defined in~\S\ref{sec:heat-kernel-below} for which $B(x_{\max}, \delta\lambda^{-1/2})\subset B(x_j, r_0/2)$.
 The aim of this subsection is to prove
\begin{theorem}\label{thm:cap-estimate}
    $$\CAP(K, U)\leq C\delta^2 (\delta\lambda^{-1/2})^{d-2}\ .$$
\end{theorem}
Theorem ~\ref{thm:cap-estimate} immediately follows from Proposition~\ref{prop:capacity-temperature} and 
the following estimate.
\begin{proposition}\label{prop:temp-upper-bound-no-lambda}
There exist $c_1, c_2>0$ such that
	\begin{equation*}
		\psi_{K, U}(\delta^2\lambda^{-1},x_{\max}) \leq c_1\delta^2
	\end{equation*}
   for all $\lambda>c_2$ and $0<\delta<1$.
\end{proposition}

First, by raising the boundary temperature on~$\partial U$ we have
\begin{lemma}\label{lem:raising-temperature}
For all $t>0$, $x\in U\setminus K$
\begin{equation*}
	\psi_{K, U}(t,x) \leq 1 - \int_{U\setminus K} p_{U\setminus K}(t, x,y)\, dy\ .
\end{equation*}
\end{lemma}
\begin{proof}
Indeed, the expression on the right hand side gives the solution to the heat equation with initial condition set to zero in $U\setminus K$ and with  boundary  values set to one on $\partial U$ and~$\partial K$. Therefore, it has same initial condition as $\psi_{K, U}(t,x)$ but higher values on the boundary. 
\end{proof}

 The elegant idea (\cite{geor-mukh}) for the proof of Proposition~\ref{prop:temp-upper-bound-no-lambda} is the majorization 
 of $\psi_{K, U}(t, x)$ for $x\in\Omega$ by the heat flow starting at $1-u_\lambda(x)/u_\lambda(x_{\max})$ with constant boundary condition:
\begin{equation}\label{ineq:elegantidea}
	e^{-\la t}u_\lambda(x_{\max}) = \int_\Omega p_\Omega(t, x_{\max},y)u_\la(y)\, dy
	\leq  u_{\lambda}(x_{\max})\int_\Omega p_\Omega(t, x_{\max},y)\, dy\ .
\end{equation}
If $\Omega$ is contained in~$U$ then we have~$\Omega\subset U\setminus K$ and  $p_\Omega\leq p_{U\setminus K}$. Thus, in this case we can immediately obtain the estimate we aim for by combining~\eqref{ineq:elegantidea} and~Lemma~\ref{lem:raising-temperature}.
 However, we need to treat  the case where $\Omega$ is not fully contained in the ball~$U$, which requires several estimates on the heat kernel.
 \begin{proof}[Proof of Proposition~\ref{prop:temp-upper-bound-no-lambda}]
  We decompose $\Omega$ into two parts:
$$\Omega=(\Omega\cap U)\cup (\Omega\setminus U)$$
Since $\Omega\cap U\subseteq U\setminus K$ we have
\begin{multline}\label{ineq:omega-uminusK-comparison}
 \int_\Omega p_\Omega(t, x_{\max},y)\, dy
	\leq \int_{U\setminus K} p_{U\setminus K}(t, x_{\max},y)\, dy \\+ \int_{\Omega\cap U} (p_\Omega (t, x_{\max},y)-p_{\Omega\cap U}(t, x_{\max},y))\, dy +  \int_{\Omega\setminus U} p_\Omega(t, x_{\max},y)\, dy\ .
\end{multline}
To find upper bounds on the last term on the right hand side of \eqref{ineq:omega-uminusK-comparison}, we recall that $p_{\Omega} \leq p_M$ and that $p_M$ obeys the Li-Yau Gaussian upper bound in Theorem~\ref{thm:heat-kernel-upper-bd}.
Since $d(x_{\max}, \Omega\setminus U)\geq r_0/2$,  we have for~$0<t<1$
\begin{equation}\label{ineq:3rd-term}
\int_{\Omega\setminus U} p_M(t, x_{\max}, y)\  dy\leq \Vol(M) t^{-d/2}e^{-Cr_0^2/t} \leq
C_1e^{-C_0r_0^2/(2t)}
\end{equation}

To bound the second term on the right hand side of~\eqref{ineq:omega-uminusK-comparison} we start by the general comparison Lemma~\ref{lem:intersection-comparison} below, and obtain
\begin{equation}\label{ineq:2nd-term-I}
\int_{\Omega\cap U} (p_\Omega (t, x_{\max},y)-p_{\Omega\cap U}(t, x_{\max},y))\, dy \leq 1-\int_U p_U(t, x_{\max}, y)\, dy\ .
\end{equation}

 In order to estimate the right hand side of \eqref{ineq:2nd-term-I}, 
 we note that according to Theorem~\ref{thm:not-feeling-bdry} we can find $\eps>0$ and $t_0>0$
 such that for all $0<t<t_0$ and $x, y\in B(x_j, 3r_0/4)$
 
\begin{equation*}
    p_U(t, x, y)\geq (1-e^{-\eps /t})p_M(t, x, y)\ .
\end{equation*}

Then, we obtain for $0<t<t_0$
\begin{equation*}
\begin{split}
\int_{U} & p_U(t, x_{\max}, y)\, dy\geq \int_{B(x_j, 3r_0/4)} p_U(t, x_{\max}, y)\, dy\\
&\geq (1-e^{-\eps/t})\int_{B(x_j, 3r_0/4)} p_M(t, x_{\max}, y)\, dy\\
&=(1-e^{-\eps/t})\left(1-\int_{M\setminus B(x_j, 3r_0/4)} p_M(t, x_{\max}, y)\, dy\right)\\
&\geq (1-e^{-\eps/t})(1-Ce^{-Cr_0^2/t}) \geq 1-e^{-\eps/t}-Ce^{-Cr_0^2/t}
\end{split}
\end{equation*}
where we have used  the upper bound in Theorem~\ref{thm:heat-kernel-upper-bd} on $p_M$.
It follows that
\begin{equation}\label{ineq:2nd-term-II}
1-\int_U p_U (t, x_{\max}, y)\, dy\leq (C+1)e^{-A/t}\ ,
\end{equation}
where $A=\{\min{\eps, Cr_0^2}\}$.

Collecting the estimates~\eqref{ineq:elegantidea},~\eqref{ineq:omega-uminusK-comparison},~\eqref{ineq:3rd-term} and~\eqref{ineq:2nd-term-II}
we get that for all $0<t<t_0$
\begin{equation*}
e^{-\lambda t} \leq\int_{\Omega} p_{\Omega}(t, x_{\max}, y)\, dy \leq \int_{U\setminus K} p_{U\setminus K}(t, x_{\max}, y)\, dy + Ce^{-A/t}\ ,
\end{equation*}
implying
$$1-\int_{U\setminus K} p_{U\setminus K} (t, x_{\max}, y)\, dy
\leq 1-e^{-\lambda t} +Ce^{-A/t}\leq 1-e^{-\delta^2}+Ce^{-\delta^{-2}}< C\delta^2\ ,$$
for all $t<\delta^2\lambda^{-1}$, $\lambda>A^{-1}$ and $0<\delta<1$.
Finally we apply Lemma~\ref{lem:raising-temperature} to get the required estimate.
\end{proof}

It remains to prove a general comparison lemma
\begin{lemma}\label{lem:intersection-comparison}
For any open sets $W_1, W_2$  and $x\in W_1\cap W_2$
	\begin{equation} \label{eqheat1}
 \int_{W_1\cap W_2} (p_{W_1} (t, x,y)-p_{W_1\cap W_2}(t, x,y))\, dy \leq 1 - \int_{W_2} p_{W_2}(t, x,y)\, dy\ .
	\end{equation}
\end{lemma}
\begin{proof}
    Both sides satisfy the heat equation in $W_1\cap W_2$, and are equal to~$0$ at time~$0$.
    Consider $$\partial(W_1\cap W_2)=\left((\partial W_1)\cap \overline{W_2}\right)\cup \left(\overline{W_1}\cap \partial W_2\right)\ .$$
    The left hand side is~$0$ on $(\partial W_1)\cap \overline{W_2}$,
    while the right hand side is~$1$ on $\overline{W_1}\cap\partial W_2$.
    Since both sides attain values only between~$0$ and $1$, it follows that
    the left hand side is not bigger than the right hand side on $\partial(W_1\cap W_2)$,
    and the inequality follows from the maximum principle.
\end{proof}

\subsection{Passing from capacity to volume}
We first recall the following basic estimate
\begin{theorem}[{\cite{mazya}*{\S 2.2.3, Corollary 2}}]
    $$ \Vol(K)\leq C\CAP (K, U)^{d/(d-2)}\ .$$
\end{theorem}

Exhausting $B(x_{\max}, \delta\lambda^{-1/2})\setminus \overline{\Omega}$ by smooth compact sets while using Theorem~\ref{thm:cap-estimate} gives that 
$$\Vol(B(x_{\max}, \delta\lambda^{-1/2})\setminus \overline{\Omega})\leq C\delta^{2d/(d-2)}(\delta\lambda^{-1/2})^{d}\ .$$

We also know by~\cite{hardt-simon} that the $d$-dimensional Hausdorff measure of the nodal set is zero.
So we have 
    $$\frac{\Vol\left(B(x_{\max}, \delta\lambda^{-1/2})\setminus\Omega\right)}{\Vol\left(B(x_{\max}, \delta\lambda^{-1/2})\right)}\leq C\delta^{2d/(d-2)}\ .$$
This concludes the proof of Theorem~\ref{thm:capacity}.

\section{Appendix: Heat kernel bounds in small times}\label{sec:appendix}
We recall celebrated Gaussian upper and lower bounds on the heat kernel,
together with the principle of not feeling the boundary.

In the following special case of Li-Yau upper bound the points~$x, y$ may be taken far apart:
\begin{theorem}[{\cite{li-yau}*{Corollary 3.1}, \cite{davies}*{Theorem 16}}]\label{thm:heat-kernel-upper-bd}
Let $M$ be a closed Riemannian manifold. For all $0<t<1$ and $x, y\in M$ we have
\begin{equation*}
p_M(t, x,y) \leq C_1 t^{-d/2} e^{-C_2 d(x,y)^2/t}
\end{equation*}
\end{theorem}

The following lower bound on closed Riemannian manifolds follows from the comparison theorem of Cheeger-Yau~\cite{cheeger-yau}*{Theorem 3.1} and the explicit formula for the heat kernel on hyperbolic space~\cite{davies-mand}*{Theorem 3.1}. For different proofs with improved dependence on the geometry see~\cite{bakry-qian}*{p. 147} and~\cite{li-xu}*{Theorem 1.5}. In fact, we apply the lower bound only for $x$ close to $y$, in which case it can be derived also from the small time asymptotic expansion of the heat kernel when~$x$ is close to~$y$ (see~\cite{BGM}*{\S III.E} or~\cite{kannai}*{formula~(1.2)}).

\begin{theorem}\label{thm:heat-kernel-lower-bd}
    Let $M$ be a closed Riemannian manifold. There exists $C_3>0$ such that
    for all $0<t<1$, $x, y\in M$
\begin{equation}
    p_M(t, x,y) \geq C_3t^{-d/2}e^{-d(x, y)^2/(4t)}
\end{equation}
\end{theorem}

To compare the Dirichlet heat kernel of a domain in a closed manifold to the heat kernel of the manifold we recall the following quantitative ``principle of not feeling the boundary'' due to Norris:
\begin{theorem}[{\cite{norr}*{proof of Theorem 1.3}}]\label{thm:not-feeling-bdry}
	Suppose $(M,g)$ is a closed  Riemannian manifold. For every strongly convex geodesic ball $B(x_*, r_0)\subset M$ of radius~$r_0$ there exist $t_0>0$ and $\eps>0$  such that
    \begin{equation}
		\frac{p_{B(x_*, r_0)}(t, x,y)}{p_{M}(t, x,y)} \geq 1 - e^{-\eps/t}
	\end{equation}
 for all $x, y\in B(x_*, 3r_0/4)$ and $0<t<t_0$.
\end{theorem}

\vspace{3ex}

\textsc{Section de math\'{e}matiques, Universit\'e de Gen\`{e}ve, 24 rue du G\'{e}n\'{e}ral Dufour,
Case postale 64,
1211 Gen\`{e}ve 4, Switzerland}

\emph{Email address: }\texttt{\textbf{philippe.charron@unige.ch}}
\vspace{3ex}

\textsc{Einstein Institute of Mathematics, Edmond J. Safra Campus,
The Hebrew University of Jerusalem, Jerusalem 9190401, Israel}

\emph{Email address: }\texttt{\textbf{dan.mangoubi@mail.huji.ac.il}}

\end{document}